\documentclass[a4paper,reqno]{amsart}
\usepackage{amssymb, amsmath, amscd}
\usepackage[final]{graphicx}
\usepackage{float}\usepackage{wrapfig}
\usepackage{color}
\usepackage{bbm}
\usepackage[final]{graphicx}

\def\mm{\mathfrak{m}}\def\kk{\text{\bf k}}

\newdimen\theight
\def\TeXref#1{%
             \leavevmode\vadjust{\setbox0=\hbox{{\tt
                     \quad\quad  {\small \textrm #1}}}%
             \theight=\ht0
             \advance\theight by \lineskip
             \kern -\theight \vbox to
             \theight{\rightline{\rlap{\box0}}%
             \vss}%
           }}%

\newtheorem{theorem}{Theorem}[section]

\newtheorem{lemma}[theorem]{Lemma}

\newtheorem{definition}[theorem]{Definition}

\newtheorem{example}[theorem]{Example}
\makeatletter
 \@addtoreset{equation}{section}
\makeatother
\begin{document}
\title[Witten Deformation]{The Witten deformation of the Dolbeault complex}
\author
{J. \'Alvarez L\'opez and P. Gilkey}
\address{PBG: Mathematics Department, University of Oregon, Eugene OR 97403-1222, USA}
\email{gilkey@uoregon.edu}
\address{JAL:  Department of Geometry and Topology, Faculty of Mathematics,
University of Santiago de Compostela, 15782 Santiago de Compostela,
Spain}
\email{jesus.alvarez@usc.es}
\thanks{Research partially supported by Projects MTM2016-75897-P and MTM2017-89686-P (AEI/FEDER, UE)}
\begin{abstract}
We introduce a Witten-Novikov type perturbation $\bar\partial_{\bar\omega}$ of the Dolbeault complex of any complex K\"ahler manifold, defined by a form $\omega$ of type $(1,0)$ with $\partial\omega=0$. We give an explicit description of the associated index density which shows that it exhibits a nontrivial dependence on $\omega$. The heat invariants of lower order are shown to be zero.
\end{abstract}
\maketitle

\section{Introduction}
\subsection{Historial Summary and motivation}

Let $M$ be a closed manifold of dimension $m$, and let $h$ be a smooth function on $M$. 
Witten~\cite{W82} introduced a perturbation of the de~Rham differential of the form $d_h=d+\operatorname{ext}(dh)$, 
using exterior multiplication by $dh$. Since $d_h$ is gauge equivalent to $d$, the Betti numbers are unchanged. 
Given a Riemannian metric $g$ on $M$, the perturbed de~Rham codifferential is $\delta_h=\delta+\operatorname{int}(dh)$
where $\operatorname{int}$ denotes interior multiplication.
In general, the perturbed Laplacian, $\Delta_h=d_h\delta_h+\delta_hd_h$, is not gauge
equivalent to $\Delta$ and thus can have a different spectrum. In fact, when $h$ is a Morse function, Witten gave a beautiful
analytical proof of the Morse inequalities by analyzing the spectrum of $\Delta_{sh}$ as $s\to\infty$. 
This program of Witten was continued by Helffer and Sj\"ostrand~\cite{HS85}, and by Bismut and Zhang~\cite{BZ92}.

With more generality, Novikov~\cite{N81,N82} defined similar perturbed operators 
$d_\iota$, $\delta_\iota$ and $\Delta_\iota$, replacing $dh$ with a real closed $1$-form $\iota$ on $M$. 
He used Witten's procedure to estimate the zeros of $\iota$ if $\iota$ is of Morse type. 
Since $d_\iota$ need not be gauge equivalent to $d$, the new twisted Betti numbers
can be different.  However, one can show that the twisted Betti numbers of $d_{s\iota}$ ($s\in\mathbb R$) 
are constant except for a
finite number of values of $s$, where the dimensions may jump. Those ground values of twisted Betti numbers
are called the Novikov numbers of the cohomology class $[\iota]$; they are used in the Novikov version of the Morse inequalities. 
We refer to related work of  Braverman and Farber~\cite{BF97} and of Pazhitnov~\cite{Pa87}, and, more recently, 
to the work of many other authors \cite{BH01,BH08,HM06,Mi15}.

In previous work \cite{AG20}, we used methods of invariance theory to prove that the local index density for the 
Witten-Novikov Laplacian $\Delta_\iota$ is the Euler form if $m$ is even and, in particular, does not depend on $\iota$. 
If $m$ is odd, the local index density vanishes. The heat trace invariants of smaller order are also trivial, but the heat trace invariants
of higher order exhibit a nontrivial dependence on $\iota$. A different proof of the invariance of the twisted index density
was also given by the first author, Kordyukov, and Leichtnam~\cite{AKL}, where it was
 applied to study certain trace formulas for foliated flows (our original motivation). In \cite{AG20}, 
 we also extended the invariance of the twisted index density to the setting of manifolds with boundary, and gave an 
 equivariant version of that invariance for maps. The situation in the complex setting is quite different. We proved that the
 local index density for a Witten-Novikov type perturbation of the Dolbeault complex exhibits non-trivial dependence on
 the twisting 1-form in the case of Riemann surfaces.

In the present paper, we extend the study of the Witten-Novikov type perturbation of the Dolbeault complex to the case of an arbitrary complex K\"ahler manifold $(M,g,J)$ of dimension $m=2\mm$. We consider the Dolbeault complex $\bar\partial$ with coefficients in an auxiliary holomorphic vector bundle $E$ over $M$ equipped with a Hermitian metric $h$. The Hirzebruch-Riemann-Roch Theorem states that its index is given by the integral on $M$ of $\{\operatorname{Td}(M,g,J)\wedge\operatorname{ch}(E,h)\}_{m}$ (the homogeneous component of degree $m$ of the product of the Todd genus of $(M,g,J)$ and the Chern character of $(E,h)$). This theorem was refined by the second author~\cite{G73,G73a}, and by Atiyah, Bott, and Patodi~\cite{ABP73}, showing that $\{\operatorname{Td}(M,g,J)\wedge\operatorname{ch}(E,h)\}_{m}$ is indeed the index density that shows up in the asymptotic expansion of the heat kernel.

In this complex setting, the Witten-Novikov deformation of the Dolbeault complex is $\bar\partial_{\bar\omega}=\bar\partial+\operatorname{ext}(\bar\omega)$, where $\omega$ is a form of type $(1,0)$ on $M$ with $\partial\omega=0$. We use methods of invariance theory to give an explicit description of its index density, which is a perturbation of $\{\operatorname{Td}(M,g,J)\wedge\operatorname{ch}(E,h)\}_{m}$ with a non trivial dependence on $\omega$. The other heat invariants of lower order are shown to be zero. As a possible application, this description might be a step in a version for the leafwise Dolbeault complex of the trace formula for foliated flows given in~\cite{AKL}.

\subsection{Operators of Laplace Type} Henceforth, let $\operatorname{dvol}$ be the measure
defined by a Riemannian metric $g$ on a closed manifold
 $M$ of dimension $m$ and let $h$ be a Hermitian fiber metric on a vector bundle $E$ over $M$.
A second order partial differential operator $D$ on $C^\infty(E)$ is said to be of
{\it Laplace type} if the leading symbol is given by the metric tensor, i.e. if
$$
D=-\left\{\sum_{i,j=1}^mg^{ij}\frac{\partial^2}{\partial x^i\partial x^j}\operatorname{id}+\sum_{k=1}^mA^k\frac\partial{\partial x^k}+B\right\}
$$
relative to a system of local coordinates $(x^1,\dots, x^m)$ for $M$ and relative to a local
frame for $E$ where $g^{ij}=g(dx^i,dx^j)$ and where $A^k$ and $B$ are endomorphisms of $E$.
The following result follows from work of Seeley~\cite{S68} and others.

\begin{theorem} Let $D$ be an operator of Laplace type. 
\begin{enumerate}
\item There exists a smooth kernel $K(t,x,y,D)$ for $t>0$ so that
$$\left\{e^{-tD}\phi\right\}(x)=\int_MK(t,x,y,D)\phi(y)\operatorname{dvol}(y)\,.$$
\item There exist local invariants $a_{m,2n}(D)(x)$
so that as $t\downarrow0$,
\begin{eqnarray*}
&&\operatorname{Tr}_{E_x}K(t,x,x,D)\sim\sum_{n=0}^\infty t^{(2n-m)/2}a_{m,2n}(D)(x),\\
&&\operatorname{Tr}_{L^2}\{e^{-tD}\}\sim\sum_{n=0}^\infty t^{(2n-m)/2}\int_Ma_{m,2n}(D)(x)\operatorname{dvol}(x)\,.
\end{eqnarray*}\end{enumerate}
\end{theorem}
\subsection{The local index density}
Let $\mathcal{E}=\{d_i:C^\infty(E_i)\rightarrow C^\infty(E_{i+1})\}$ where $(E_i,h_i)$ is a finite collection of
Hermitian vector bundles and where the $d_i$ are first order partial differential operators. We shall
say that $\mathcal{E}$
is an {\it elliptic complex of Dirac type} if $d_{i+1}d_i=0$ and if the associated self-adjoint second order operators
$D_i:=d_i^*d_i+d_{i-1}d_{i-1}^*$ are of Laplace type. The cohomology groups of $\mathcal{E}$ are given by
$$
H^i(\mathcal{E}):=\frac{\operatorname{kernel}(d_i:C^\infty(E_i)\rightarrow C^\infty(E_{i+1}))}
{\operatorname{image}(d_{i-1}:C^\infty(E_{i-1})\rightarrow C^\infty(E_i))}\,.
$$
The Hodge Decomposition Theorem permits us to identify $H^i(\mathcal{E})$ with $\ker(D_i)$; these
groups are finite dimensional and we take the super-trace to define
$$
a_{m,2n}(\mathcal{E}):=\sum_i(-1)^ia_{m,2n}(D_i)\text{ and }
\operatorname{index}(\mathcal{E}):=\sum_i(-1)^i\dim H^i(\mathcal{E})\,.
$$
If $m$ is odd, then $\operatorname{index}(\mathcal{E})=0$ so we assume $m$ even henceforth.
The invariant $a_{m,2n}(\mathcal{E})$ for $2n=m$ is called the {\it local index density} as
a cancellation argument due to Bott shows that
$$
\int_Ma_{m,2n}(\mathcal{E})(x)\operatorname{dvol}(x)=\left\{\begin{array}{cl}\operatorname{index}(\mathcal{E})&\text{if }2n=m\\
0&\text{if }2n\ne m\end{array}\right\}\,.
$$

\subsection{The de Rham complex} 
Let $\Lambda^i(M)$ be the vector bundle of $i$-forms and let
$$
d:C^\infty(\Lambda^i(M))\rightarrow C^\infty(\Lambda^{i+1}(M))\text{ for }0\le i\le m-1\,,
$$
be exterior differentiation. This defines an elliptic complex of Dirac type we shall denote by
$\mathcal{E}_{\operatorname{deR}}(M,g)$.
Let $\chi(M)$ be the Euler-Poincar\'e characteristic of $M$.
The Hodge--de Rham theorem permits us to identify $H^{p}(\mathcal{E}_{\operatorname{deR}}(M,g))$
with the topological cohomology groups $H^p(M;\mathbb{C})$ and shows that
$\operatorname{index}(\mathcal{E}_{\operatorname{deR}}(M,g))=\chi(M)$.

\subsection{The Dolbeault complex}
Let $J$ be an integrable almost complex structure on a smooth manifold $M$ of complex dimension $\mm$
 and corresponding real dimension $m=2\mm$.
 Let $g$ be a $J$ invariant
 Riemannian metric on $M$; $(M,g,J)$ is a Hermitian holomorphic manifold. 
 Let $\Omega(X,Y):=g(X,JY)$ be the K\"ahler form; we say $(M,g,J)$ is {\it K\"ahler} if $d\Omega=0$.
 Let $E$ be an auxiliary holomorphic
 vector bundle over $M$ equipped with a Hermitian metric $h$. 
 Let $\sqrt2\ \bar\partial$ be the normalized Dolbeault operator;
 the normalizing constant of $\sqrt{2}$ is present to ensure 
that this is of Dirac type. 
 The Dolbeault complex $\mathcal{E}_{\operatorname{Dol}}(M,g,J,E,h)$ is defined by
 $$
\sqrt{2}\bar\partial:C^\infty(\Lambda^{0,i}(M)\otimes E)\rightarrow
C^\infty(\Lambda^{0,i+1}(M)\otimes E)\text{ for }0\le i\le\mm-1\,.
$$
Let $H^p(M;\mathcal{O}(E))$
be the cohomology groups of $M$ with coefficients in the sheaf of holomorphic sections to $E$. Identify
$H^p(\mathcal{E}_{\operatorname{Dol}}(M,g,J,E,h))$ with $H^p(M;\mathcal{O}(E))$;
if $E$ is the trivial line bundle, then
$\operatorname{index}\{\mathcal{E}_{\operatorname{Dol}}(M,g,J,E,h)\}$ is the arithmetic genus of $M$.

\subsection{The Chern-Gauss-Bonnet and Hirzebruch-Riemann-Roch Theorems}
Let $m=2\mm$. Let
$R_{ijkl}$ denote the components of the curvature tensor relative to a local orthonormal frame for the tangent bundle of $M$.
We follow the discussion in Chern~\cite{C44} and define the Pfaffian or Euler form by setting:
\begin{eqnarray*}
&&\operatorname{Pf}_m(x,g):=\sum_{i_1,\dots,i_m,j_1,\dots,j_m=1}^m\frac{(-1)^{\mm}}{8^{\mm}\pi^{\mm}\mm!}g(e^{i_1}\wedge\dots\wedge e^{i_m},e^{j_1}\wedge\dots\wedge e^{j_m})\\
&&\qquad\qquad\qquad\qquad\qquad\qquad\qquad R_{i_1i_2j_1j_2}\dots R_{i_{m-1}i_mj_{m-1}j_m}(x)\,.
\end{eqnarray*}
In the complex setting, let $\operatorname{Td}(M,g,J)$ be the 
total Todd genus of the complex tangent bundle of $(M,g,J)$
and let $\operatorname{ch}(E,h)$ be the total Chern character; we refer to Hirzebruch~\cite{H66} for details. We set
$$
\{\operatorname{Td}(M,g,J)\wedge\operatorname{ch}(E,h)\}_{m}=\sum_{i+j=m}
\operatorname{Td}_i(M,g,J)\wedge\operatorname{ch}_j(E,h)\,.
$$
We use the Hodge $\star$ operator to identify top dimensional forms with scalar functions. 
We refer to  Chern~\cite{C44} for the proof of Assertion~(1) and to
Hirzebruch~\cite{H66} for the proof of Assertion~(2) in the following result.
\begin{theorem}\label{T1.2}\ \begin{enumerate}
\item $\displaystyle\operatorname{index}(\mathcal{E}_{\operatorname{deR}}(M,g))=
\int_M\operatorname{Pf}_m(M,g)\operatorname{dvol}$.
\item $\displaystyle\operatorname{index}(\mathcal{E}_{\operatorname{Dol}}(M,g,J,E,h))=
\int_M\star\{\operatorname{Td}(M,g,J)\wedge\operatorname{ch}(E,h)\}_{m}\operatorname{dvol}$.
\end{enumerate}\end{theorem}

By identifying the local index densities of the de Rham and
Dolbeault complexes with the integrands of Theorem~\ref{T1.2},
Patodi~\cite{P71,P71a} gave a heat equation proof of Theorem~\ref{T1.2}~(1)
in the real setting
and of Theorem~\ref{T1.2}~(2)  in the complex K\"ahler setting by showing:

\begin{theorem}\label{T1.3}
\ \begin{enumerate}
\item $\displaystyle
a_{m,2n}(\mathcal{E}_{\operatorname{deR}}(M,g))=\left\{\begin{array}{cl}
0&\text{if }2n<m\\\operatorname{Pf}_m(M,g)&\text{if }2n=m
\end{array}\right\}$.
\smallbreak\item If $(M,g,J)$ is K\"ahler, then $a_{m,2n}(\mathcal{E}_{\operatorname{Dol}}(M,g,J,E,h))$
$$=\left\{\begin{array}{cl}
0&\text{if }2n<m\\
\star\{\operatorname{Td}(M,g,J)\wedge\operatorname{ch}(E,h)\}_{m}&\text{if }2n=m\end{array}\right\}\,.
$$
\end{enumerate}\end{theorem}
Shortly thereafter, other proofs of Theorem~\ref{T1.3} were given.
Gilkey~\cite{G73,G73a} used invariance theory directly and Atiyah, Bott, and Patodi~\cite{ABP73}
combined invariance theory with a study of the twisted signature complex and the twisted
spin-c complex to prove Theorem~\ref{T1.3}. 
The subject has an extensive history
and we refer to \cite{G95} for further details. 
As noted, the twisted signature complex and
twisted spin complex can be treated using heat equation methods \cite{ABP73,G73} and
a heat equation proof of the full Atiyah-Singer index theorem given thereby.
We note that the local index density for the Dolbeault complex
does not agree in general with the Hirzebruch-Riemann-Roch
integrand of Theorem~\ref{T1.2}~(2) in the non-K\"ahler setting as was shown in later work by 
Gilkey,  Nik\v cevi\'c,\ and Pohjanpelto~\cite{GNP97};
Theorem~\ref{T1.3}~(2) can fail if $(M,g,J)$ is not assumed K\"ahler.

\subsection{The Witten deformation} If $\iota$ is a closed 1-form on a
Riemannian manifold $(M,g)$, one can
define the deformed de Rham complex $\mathcal{E}_{\operatorname{deR}}(M,g,\iota)$ by setting 
$$
d_\iota:=d+\operatorname{ext}(\iota):C^\infty(\Lambda^iM)\rightarrow C^\infty(\Lambda^{i+1}(M))\,.
$$
The assumption that $\iota$ is closed ensures that $d_\iota^2=0$; since $\iota$ introduces a 
$0^{\operatorname{th}}$ perturbation, the leading symbol of the associated second order operators is
unchanged so $\mathcal{E}_{\operatorname{deR}}(M,g,\iota)$ is an elliptic complex of Dirac type.
The authors~\cite{AG20} showed previously that the deformation $\iota$ does not enter in the index density
in this setting
$$
a_{m,2n}(\mathcal{E}_{\operatorname{deR}}(M,g,\iota))
=\left\{\begin{array}{cl}0&\text{if }2n<m\\\operatorname{Pf}_m(M,g)&\text{if }2n=m\end{array}\right\}\,.
$$
This result is sharp; if $2n>m$, then $a_{m,2n}(\mathcal{E}_{\operatorname{deR}}(M,g,\iota))$ 
does depend upon $\iota$ in general.

In the complex setting, let $\omega$ be a form of type $(1,0)$ on $M$ with $\partial\omega=0$ and let
$\mathcal{M}:=(M,g,J,\omega,E,h)$. Set 
$\bar\partial_{\bar\omega}:=\bar\partial+\operatorname{ext}(\bar\omega)$. The assumption 
$\partial\omega=0$ implies $\bar\partial\bar\omega=0$ and ensures 
$\bar\partial_{\bar\omega}^2=0$. Let $\mathcal{E}_{\operatorname{Dol}}(\mathcal{M})$ be
the Witten perturbation of the Dolbeault complex with coefficients in $E$ defined by taking
$$
\sqrt{2}\ \bar\partial_{\bar\omega}:
C^\infty(\Lambda^{(0,i)}\otimes E)\rightarrow C^\infty(\Lambda^{(0,i+1)}\otimes E)
\text{ for }0\le i\le\mm-1\,.
$$
This is an elliptic complex of Dirac type. Let $\Im(\omega)=\frac1{2\sqrt{-1}}(\omega-\bar\omega)$
be the imaginary part of $\omega$. Let
\begin{eqnarray*}
&&\Theta:=\displaystyle\sum_k\frac{1}{k!\pi^k}\{d\Im(\omega)\}^k,\\
&&\left\{\operatorname{Td}\wedge\operatorname{ch}\wedge\Theta\right\}_{m}:=
\sum_{i+j+k=m}\operatorname{Td}(M,g,J)_i\wedge\operatorname{ch}(E,h)_j\wedge\Theta_k\,.
\end{eqnarray*}
The following is the main new result of this paper.

\begin{theorem}\label{T1.4}
$a_{m,2n}(\mathcal{E}_{\operatorname{Dol}}(\mathcal{M}))=\left\{\begin{array}{cl}0&\text{if }2n<m\\
\star\left\{\operatorname{Td}\wedge\operatorname{ch}\wedge\Theta\right\}_{m}&\text{if }2n=m\end{array}\right\}$.\end{theorem}

\subsection{The signature and spin complexes}
Let $M$ be an oriented manifold of dimension $4k$. Let
$d+\delta:C^\infty(\Lambda^\pm(M))\rightarrow C^\infty(\Lambda^\mp(M))$ be the Hirzebruch
signature complex. We then have $d_\iota +\delta_\iota =d+\delta+(\operatorname{ext}+\operatorname{int})(\iota )$.
Now $(\operatorname{ext}-\operatorname{int})(\iota ):\Lambda^\pm\rightarrow\Lambda^\mp$ but
$(\operatorname{ext}+\operatorname{int})(\iota )$ does not have this property if $\iota\ne0$. 
So $d_\iota +\delta_\iota $ does not induce a map on the signature complex; 
it is not possible to deform the signature complex in this fashion. Similarly the spin
complex can not be deformed in this fashion. The de Rham and Dolbeault complexes are 
$\mathbb{Z}$ graded and this seems to be crucial in studying the Witten deformation; 
the signature and spin complexes, on the other hand, are $\mathbb{Z}_2$ graded and
this makes all the difference. For this reason, we shall not follow the approach of 
Atiyah, Bott, and Patodi~\cite{ABP73} to study
the Witten deformation of the Dolbeault complex by passing to the spin-c complex. Instead, we shall return
to the original treatment of Gilkey~\cite{G73a} and apply invariance theory directly.
\subsection{Brief guide to the paper} In Section~\ref{S2}, we discuss product formulas for the heat trace asymptotics. 
In Section \ref{S3}, we normalize the systems of coordinates and vector bundle frames 
to be considered up to arbitrarily high, but finite, order; this in effect reduces the structure group
to the unitary group. In Section~\ref{S4}, we introduce the requisite spaces of invariants; 
the precise notion of what is meant by a ``local invariant" 
or a ``local formula" is crucial to our study. 
In Section \ref{S5}, we discuss the restriction map.
In Section~\ref{S6}, we use invariance under the action of the unitary group
$U(\mm)$
to establish certain technical results. In Section~\ref{S7}, we complete the proof of Theorem~\ref{T1.4}.

\section{Product formulas}\label{S2}
The following observations are well known -- see, for example, the discussion in Gilkey~\cite{G95}. Let $M=M_1\times M_2$ where $M_i$ are closed manifolds of dimension $m_i$.
Let $\pi_i:M\rightarrow M_i$ be projection on the $i^{\operatorname{th}}$ factor. Let $g_i$ be Riemannian
metrics on $M_i$ and let
$g=\pi_1^*g_1+\pi_2^*g_2$ be the associated Riemannian metric on $M$. Let $(E_i,h_i)$ be Hermitian
vector bundles over $M_i$ and let
$E:=\pi_1^*E_1\otimes\pi_2^*E_2$ and $h:=\pi_1^*h_1\otimes\pi_2^*h_2$
define the associated vector bundle and Hermitian inner product over $M$. Let 
$D=D_1\otimes\operatorname{id}+\operatorname{id}\otimes D_2$ be an operator of Laplace type
on $C^\infty(E)$ over $M$
where $D_i$ are operators of Laplace type on $C^\infty(E_i)$ over $M_i$. 
We have $e^{-tD}=e^{-tD_1}\otimes e^{-tD_2}$ and the associated kernel function is given by
$$
K(t,(x_1,x_2),(y_1,y_2),D)=K(t,x_1,y_1,D_1)\otimes K(t,x_2,y_2,D_2)\,.
$$
We multiply the resulting asymptotic expansions for the heat kernels
and equate coefficients of $t$ to obtain corresponding local
expressions
\begin{equation}\label{E2.a}
a_{m,2n}(D)(x_1,x_2)=\sum_{n_1+n_2=n}a_{m_1,2n_1}(D_1)(x_1)\cdot a_{m_2,2n_2}(D_2)(x_2)\,.
\end{equation}
Let $\mathcal{E}_1$ and $\mathcal{E}_2$ be elliptic complexes of Dirac type over $M_1$ and $M_2$,
respectively. The elliptic complex of Dirac type $\mathcal{E}:=\mathcal{E}_1\otimes\mathcal{E}_2$
over $M$ is defined by setting 
\begin{eqnarray*}
&&E_k=\oplus_{i+j=k}\pi_1^*(E_{1,i})\otimes\pi_2^*(E_{2,j}),\\
&&d_k=\oplus_{i+j=k}d_{1,i}\otimes\operatorname{id}+(-1)^i\operatorname{id}\otimes d_{2,j}\,;
\end{eqnarray*}
the factor of $(-1)^i$ is present to ensure $d^2=0$.
The associated operators of Laplace type then take the form
$D_k=\oplus_{i+j=k}D_{1,i}\otimes\operatorname{id}+\operatorname{id}\otimes D_{2,j}$
and consequently taking the super trace and applying Equation~(\ref{E2.a}) yields
\begin{equation}\label{E2.b}
a_{m,2n}(\mathcal{E})(x_1,x_2)
=\sum_{n_1+n_2=n}a_{m_1,2n_1}(\mathcal{E}_1)(x_1)\cdot a_{m_2,2n_2}(\mathcal{E}_2)(x_2)\,.
\end{equation}

We now turn to the Dolbeault complex. Let $\mathcal{M}_1=(M_1,g_1,J_1,E_1,h_1,\omega_1)$ and
$\mathcal{M}_2=(M_2,g_2,J_2,E_2,h_2,\omega_2)$ be given. Let 
$$
M:=M_1\times M_2,\ g:=\pi_1^*g_1+\pi_2^*g_2,\ E:=\pi_1^*E_1\otimes\pi_2^*E_2,\ h:=\pi_1^*h_1\otimes\pi_2^*h_2
$$
be as given above. Let
$\omega:=\pi_1^*\omega_1+\pi_2^*\omega_2$. We have that
$$
J:=\pi_1^*J_1\oplus\pi_2^*J_2\text{ on }T(M)=\pi_1^*T(M_1)\oplus\pi_2^*T(M_2)
$$
is an integrable almost complex structure on $M$; the auxiliary bundle $E$ is then holomorphic. We 
set $\mathcal{M}=\mathcal{M}_1\times\mathcal{M}_2=(M,g,J,E,h,\omega)$
and obtain that
$$\mathcal{E}_{\operatorname{Dol}}(\mathcal{M})=
\mathcal{E}_{{\operatorname{Dol}}}(\mathcal{M}_1)\otimes
\mathcal{E}_{{\operatorname{Dol}}}(\mathcal{M}_2)\,.$$
Equation~(\ref{E2.b})
then yields a corresponding decomposition of the local heat trace invariants
\begin{equation}\label{E2.c}
a_{m,2n}(\mathcal{E}_{\operatorname{Dol}}(\mathcal{M}))
\displaystyle=\sum_{n_1+n_2=n}a_{m_1,2n_1}(\mathcal{E}_{\operatorname{Dol}}(\mathcal{M}_1))\ 
a_{m_2,2n_2}(\mathcal{E}_{\operatorname{Dol}}(\mathcal{M}_2))
\end{equation}

\section{Normalizing the coordinates and the local frame}\label{S3}
Let $\vec z=(z^1,\dots z^\mm)$ where $z^\alpha=x^\alpha+\sqrt{-1}y^\alpha$ is a system of local holomorphic coordinates on $M$. 
Let
$$
\frac\partial{\partial_{z^\alpha}}:=\frac12\left(\frac{\partial}{\partial x^\alpha}-\sqrt{-1}\frac{\partial}{\partial y^\alpha}\right)
\text{ and }\frac\partial{\partial_{\bar z^\alpha}}:=
\frac12\left(\frac{\partial}{\partial x^\alpha}+\sqrt{-1}\frac{\partial}{\partial y^\alpha}\right)\,.
$$
We extend $g$ to a symmetric bilinear form on the complex tangent bundle; the condition that $g$ is $J$
invariant, then yields 
$$
g\left(\frac\partial{\partial z^\alpha},\frac\partial{\partial z^\beta}\right)=
g\left(\frac\partial{\partial\bar z^\alpha},\frac\partial{\partial\bar z^\beta}\right)=0\quad\text{so}\quad
g_{\alpha\bar\beta}:=g\left(\frac\partial{\partial z^\alpha},\frac\partial{\partial\bar z^\beta}\right)
$$
defines a positive definite Hermitian form. Introduce formal variables
$$
g_{\alpha_0\bar\beta_0/\alpha_1\dots\alpha_j\bar\beta_1\dots\bar\beta_k}:=
\frac\partial{\partial_{z^{\alpha_1}}}\dots\frac\partial{\partial_{z^{\alpha_j}}}
\frac\partial{\partial_{\bar z^{\beta_1}}}\dots\frac\partial{\partial_{\bar z^{\beta_k}}}
g\left(\frac\partial{\partial_{z^{\alpha_0}}},\frac\partial{\partial_{\bar z^{\beta_0}}}\right)
$$
for the holomorphic
and anti-holomorphic derivatives of the components of $g$ where there are no holomorphic derivatives if $j=0$
and no anti-holomorphic derivatives if $k=0$. We may express the K\"ahler form $\Omega(X,Y)=g(X,JY)$ as
$$
\Omega=\frac1{2\sqrt{-1}}
\sum_{\alpha_0=1}^\mm\sum_{\beta_0=1}^{\mm}g_{\alpha_0\bar\beta_0}dz^{\alpha_0}\wedge d\bar z^{\beta_0}\,.$$
We say $(M,g,J)$ is {\it K\"ahler} if $d\Omega=0$ and we impose this condition henceforth.
This condition is equivalent to the symmetries:
$$
g_{\alpha_0\bar\beta_0/\alpha_1}=g_{\alpha_1\bar\beta_0/\alpha_0}\text{ and }
g_{\alpha_0\bar\beta_0/\bar\beta_1}=g_{\alpha_0\bar\beta_1/\bar\beta_0}\,.
$$
We can differentiate these relations to see the variables $g_{\alpha_0\bar\beta_0/\alpha_1\dots\alpha_j\bar\beta_1\dots\bar\beta_k}$ are
symmetric in $\{\alpha_0\dots\alpha_j\}$ and in $\{\beta_0\dots\beta_k\}$. Set
\begin{equation}\label{E3.a}
g_{(\alpha_0\dots\alpha_j;\bar\beta_0\dots\bar\beta_k)}:=g_{\alpha_0\bar\beta_0/\alpha_1\dots\alpha_j\bar\beta_1\dots\bar\beta_k}\,.
\end{equation}
If $\sigma_1$ and $\sigma_2$ are permutations of $j+1$ and $k+1$ indices, respectively, then
\begin{equation}\label{E3.b}
g_{(\alpha_0\dots\alpha_j;\bar\beta_0\dots\bar\beta_k)}=g_{(\alpha_{\sigma_1(0)}\dots\alpha_{\sigma_1(j)};\bar\beta_{\sigma_2(0)}\dots\bar\beta_{\sigma_2(k)})}\,.
\end{equation}
Similarly introduce formal variables
\begin{equation}\label{E3.c}
h_{(p\bar q;\alpha_1\dots\alpha_j;\bar\beta_1\dots\bar\beta_k)}:=h_{p\bar q/\alpha_1\dots\alpha_j\bar\beta_1\dots\bar\beta_k}
\end{equation}
for the derivatives of
the components of the Hermitian metric on $E$.  If $\sigma_3$ and $\sigma_4$ are permutations of $j$ and $k$ indices, respectively, then
\begin{equation}\label{E3.d}
h_{(p\bar q;\alpha_1\dots\alpha_j;\bar\beta_1\dots\bar\beta_k)}
=h_{(p\bar q;\alpha_{\sigma_3(1)}\dots\alpha_{\sigma_3(j)};\bar\beta_{\sigma_4(1)}\dots\bar\beta_{\sigma_4(k)})}.
\end{equation}
Since $\partial\omega=0$, $\omega_{\alpha_0/\alpha_1}=\omega_{\alpha_1/\alpha_0}$ and similarly we set
\begin{equation}\label{E3.e}\begin{array}{l}
\omega_{(\alpha_0\dots\alpha_j;\bar\beta_1\dots\bar\beta_k)}
=\omega_{\alpha_0/\alpha_1\dots\alpha_j\bar\beta_1\dots\bar\beta_k},\\[0.05in]
\bar\omega_{(\alpha_1\dots\alpha_j;\bar\beta_0\dots\bar\beta_k)}
=\bar\omega_{\bar\beta_0/\alpha_1\dots\alpha_j\bar\beta_1\dots\bar\beta_k}.
\end{array}\end{equation}
If $\sigma_5$ is a permutation of $j+1$ indices, $\sigma_6$ is a permutation of $k$ indices, 
$\sigma_7$ is a permutation of $j$ indices, and $\sigma_8$ is a permutation of $k+1$ indices, then
\begin{equation}\label{E3.f}
\begin{array}{l}
\omega_{(\alpha_0\dots\alpha_j;\bar\beta_1\dots\bar\beta_k)}
=\omega_{(\alpha_{\sigma_5(0)}\dots\alpha_{\sigma_5(j)};\bar\beta_{\sigma_6(1)}\dots\bar\beta_{\sigma_6(k)})},\\[0.05in]
\bar\omega_{(\alpha_1\dots\alpha_j;\bar\beta_0\dots\bar\beta_k)}
=\bar\omega_{(\alpha_{\sigma_7(1)}\dots\alpha_{\sigma_7(j)};\bar\beta_{\sigma_8(0)}\dots\bar\beta_{\sigma_8(k)})}.
\end{array}\end{equation}

Lemma~2 of \cite{G73a} yields the following result.
\begin{lemma}\label{L3.1}
 Let $\mathcal{M}=(M,g,J,\omega,E,h)$. Fix a point $z_0\in M$ and a positive integer $N$. There is a holomorphic coordinate system $\vec z=(z^1,\dots,z^\mm)$
centered at $z_0$ and a holomorphic frame $\vec e$ for $E$ defined near $z_0$ so that
\begin{enumerate}
\item $g_{\alpha\bar\beta}(z_0)=\delta_{\alpha\beta}$ and $h_{p\bar q}(z_0)=\delta_{p,q}$.
\item $g_{(\alpha_0\dots\alpha_j;\bar\beta_0)}(z_0)=g_{(\alpha_0;\bar\beta_0\dots\bar\beta_k)}(z_0)=0$ for
$1\le j,k\le N$.
\item  $h_{(p\bar q;\alpha_1\dots\alpha_j)}(z_0)=h_{(p\bar q;\bar\beta_1\dots\bar\beta_k)}(z_0)=0$
for $1\le j,k\le N$.
\end{enumerate}\end{lemma}
In other words, these variables vanish at the basepoint if 
either there are no holomorphic or there are no anti-holomorpic derivatives.
With these normalizations, the variables $g_{(\cdot;\cdot)}$, $h_{(p\bar q;\cdot;\cdot)}$,
$\omega_{(\cdot;\cdot)}$, and $\bar\omega_{(\cdot;\cdot)}$ are tensorial; we have reduced
the structure group to the unitary groups $U(\mm)$ and $U(\dim(E))$ modulo transformations of order $O(|z|^{N+1})$.

\section{Spaces of invariants}\label{S4}
We must be rather precise in what is meant by a ``local invariant" or a ``local formula". We do this as follows.

\subsection{The algebra $\mathfrak{A}$} We introduce the polynomial algebra $\mathfrak{A}_\mm$
in the variables of Equations~(\ref{E3.a}), (\ref{E3.c}), and
(\ref{E3.e}) and impose the relations of Equations~(\ref{E3.b}), (\ref{E3.d}), (\ref{E3.f}), and
Lemma~\ref{L3.1}:
\medbreak\qquad\quad
$\mathfrak{A}_\mm:=\mathbb{C}[g_{(\alpha_0\dots\alpha_{j_1};\bar\beta_0\dots\bar\beta_{k_1})},
h_{(p\bar q;\alpha_1\dots\alpha_{j_2};\bar\beta_1\dots\bar\beta_{k_2})},
\omega_{(\alpha_0\dots\alpha_{j_3};\bar\beta_1\dots\bar\beta_{k_3})},$
\medbreak\qquad\qquad\qquad\quad$
\bar\omega_{(\alpha_1\dots\alpha_{j_4};\bar\beta_0\dots\bar\beta_{k_4})}]$
for $j_1\ge1$, $k_1\ge 1$, $j_2\ge1$, $k_2\ge 1$,
\medbreak\qquad\qquad\qquad\quad$j_3\ge0$, $k_3\ge0$, $j_4\ge0$, and $k_4\ge0$.
\medbreak\noindent
If $k_3=0$, there are no anti-holomorphic derivatives of $\omega$ and if $j_4=0$, there are no
anti-holomorphic derivatives of $\bar\omega$.
We introduce the complex dimension $\mm$ into the notation as it plays
an important role; we suppress the fiber dimension of $E$ in the interests of notational simplicity. 

If $P\in\mathfrak{A}_\mm$ and if $A$ is a monomial,
we let $c(A,P)$ be the coefficient of $A$
in $P$ and express $P=\sum_Ac(A,P)P$. We say {\it $A$ is a monomial of $P$} or that
{\it $A$ appears in $P$} if $c(A,P)\ne0$. The maps $P\rightarrow c(A,P)$ are linear maps from $\mathfrak{A}_\mm$ to
$\mathbb{C}$.

\subsection{The weight} To count the number of derivatives of $g$ and $h$, we set:
$$
\operatorname{weight}\{g_{(\alpha_0\dots\alpha_j;\bar\beta_0\dots\bar\beta_k)}\}
=\operatorname{weight}\{h_{(p\bar q;\alpha_1\dots\alpha_j;\bar\beta_1\dots\bar\beta_k)}\}:=j+k\,.
$$
Since the 1-forms $\omega$ and $\bar\omega$ appear as $0^{\operatorname{th}}$ perturbations of
$\bar\partial$ and the adjoint $\bar\partial^*$, we give $\omega$ and $\bar\omega$ weight 1 and
define
$$
\operatorname{weight}\{\omega_{(\alpha_0\dots\alpha_j;\bar\beta_1\dots\bar\beta_k)}\}=
\operatorname{weight}\{\bar\omega_{(\alpha_1\dots\alpha_j;\bar\beta_0\dots\bar\beta_k)}\}:=1+j+k\,.
$$
Distinguish the variables of weight 1 and set
$$
\Xi_{(\alpha_1\dots\alpha_e;\bar\beta_1\dots\bar\beta_f)}:=\omega_{\alpha_1}\dots\omega_{\alpha_e}\bar\omega_{\bar\beta_1}\dots\bar\omega_{\bar\beta_f}
$$
where there are no holomorphic indices if $e=0$, no anti-holomorphic indices if $f=0$, and $\Xi=1$ if $e=f=0$; we set
$
\operatorname{weight}\{\Xi_{(\alpha_1\dots\alpha_e;\bar\beta_1\dots\bar\beta_f)}\}:=e+f\,.
$

\subsection{Monomials}
Let $U_*$ and $\bar V_*$ be (possibly empty) collections of holomorphic and anti-holomorphic indices, 
respectively. If $A$ is a monomial, we express
\begin{equation}\label{E4.a}\begin{array}{l}
A=g_{(U_1;\bar V_1)}\dots g_{(U_a;\bar V_a)}\ h_{(p_1\bar q_1;U_{a+1};\bar V_{a+1})}\dots 
h_{(p_b\bar q_b;U_{a+b};\bar V_{a+b})}\\[0.05in]
\quad\cdot\omega_{(U_{a+b+1};\bar V_{a+b+1})}\dots\omega_{(U_{a+b+c};\bar V_{a+b+c})}\ 
\bar\omega_{(U_{a+b+c+1};\bar V_{a+b+c+1})}\dots\\[0.05in]
\quad\cdot\bar\omega_{(U_{a+b+c+d};\bar V_{a+b+c+d})}\Xi_{(U_{a+b+c+d+1};\bar V_{a+b+c+d+1})}\,.
\end{array}\end{equation}
In this expression, the variables in $g_*$, $h_*$, $\omega_*$, and $\bar\omega_*$ 
have weight at least 2; we distinguish the variables of weight 1 separately in $\Xi$;
since we have imposed the relations of Lemma~\ref{L3.1}, all the variables defining $\mathfrak{A}_\mm$
have positive weight.

We extend the notion of weight to be the sum of the weights of the variables
comprising $A$. If $A$ has the form given in Equation~(\ref{E4.a}), then
\begin{eqnarray*}
\operatorname{weight}(A)&=&\sum_{i=1}^a\operatorname{weight}\{g_{(U_i;\bar V_i)}\}
+\sum_{i=a+1}^{a+b}\operatorname{weight}\{h_{(p_i\bar q_i;U_{a+b};\bar V_{a+b})}\}\\
&+&\sum_{i=a+b+1}^{a+b+c}\operatorname{weight}\{\omega_{(U_i;\bar V_i)}\}
+\sum_{i=a+b+c+1}^{a+b+c+d}\operatorname{weight}\{\bar\omega_{(U_i;\bar V_i)}\}+e+f\,.
\end{eqnarray*}
We say that a polynomial $P$ is {\it homogeneous of weight $2n$} if all the monomials of $P$
have weight $2n$.

\subsection{The length} If $A$ has the form given in Equation~(\ref{E4.a}), we define  the  length $\ell(A)$ by setting
$$
\ell(A):=\left\{\begin{array}{cl}a+b+c+d&\text{if }\Xi=1\\a+b+c+d+1&\text{if }\Xi\ne1\end{array}\right\}\,.
$$

\subsection{Local invariants}
If $\vec z$ is a normalized coordinate system on $M$, if $\vec s$ is a normalized local holomorphic frame
for $E$, and if $P\in\mathfrak{A}_\mm$, we evaluate $P(\mathcal{M})(z_0)(\vec z,\vec s)$ in the obvious fashion.
If $P(\mathcal{M})(z_0):=P(\mathcal{M})(z_0)(\vec z,\vec s)$
is independent of the particular normalized coordinate system $\vec z$ and normalized frame $\vec s$
for any $\mathcal{M}$ and any $z_0$, then we shall
say that $P$ is {\it invariant}. The scalar curvature $\tau$ and the heat trace asymptotics 
$a_{m,2n}(\mathcal{E}_{\operatorname{Dol}})$ are invariant.
\begin{definition}\rm
Let $\mathcal{P}_{\mm,2n}$ be the subspace of $\mathfrak{A}_\mm$
of invariant polynomials which are homogeneous of weight $2n$; there are no invariant polynomials of odd weight.
\end{definition}

\begin{example}\rm The scalar curvature $\tau$ is an element
of $\mathcal{P}_{\mm,2}$ since $\tau$ is linear in the 2-jets of the metric and
quadratic in the 1-jets of the metric with coefficients which are smooth functions of the metric tensor. 
\end{example}

The following observation follows from the explicit combinatorial algorithm given by Seeley~\cite{S68} 
for computing the heat trace invariants.
\begin{lemma}\label{L4.3}
$a_{m,2n}(\mathcal{E}_{\operatorname{Dol}})\in\mathcal{P}_{\mm,2n}$.
\end{lemma}

\section{The restriction map}\label{S5}

\subsection{The degree} Let $\operatorname{deg}_\alpha$ and $\operatorname{deg}_{\bar\beta}$ be the
total number
of times the index $\alpha$ or $\bar\beta$ appears in one of the variables comprising $\mathfrak{A}$:
\smallbreak\quad$\displaystyle\operatorname{deg}_\alpha\{g_{(\alpha_0\dots\alpha_j;\bar\beta_0\dots\bar\beta_k)}\}
=\sum_{\nu=0}^j\delta_{\alpha\alpha_\nu}$,\quad\ \ \ 
$\displaystyle\operatorname{deg}_{\bar\beta}\{g_{(\alpha_0\dots\alpha_j;\bar\beta_0\dots\bar\beta_k)}\}
=\sum_{\nu=0}^k\delta_{\bar\beta\bar\beta_\nu}$,
\smallbreak\quad$\displaystyle\operatorname{deg}_\alpha\{h_{(p\bar q;\alpha_1\dots\alpha_j;\bar\beta_1\dots\bar\beta_k)}\}
=\sum_{\nu=1}^j\delta_{\alpha\alpha_\nu}$,\quad
$\displaystyle\operatorname{deg}_{\bar\beta}\{h_{(p\bar q;\alpha_1\dots\alpha_j;\bar\beta_1\dots\bar\beta_k)}\}
=\sum_{\nu=1}^k\delta_{\bar\beta\bar\beta_\nu}$,
\smallbreak\quad$\displaystyle\operatorname{deg}_\alpha\{\omega_{(\alpha_0\dots\alpha_j;\bar\beta_1\dots\bar\beta_k)}\}
=\sum_{\nu=0}^j\delta_{\alpha\alpha_\nu}$,\quad\ \ 
$\displaystyle\operatorname{deg}_{\bar\beta}\{\omega_{(\alpha_0\dots\alpha_j;\bar\beta_1\dots\bar\beta_k)}\}
=\sum_{\nu=1}^k\delta_{\bar\beta\bar\beta_\nu}$,
\smallbreak\quad$\displaystyle\operatorname{deg}_\alpha\{\bar\omega_{(\alpha_1\dots\alpha_j;\bar\beta_0\dots\bar\beta_k)}\}
=\sum_{\nu=1}^j\delta_{\alpha\alpha_\nu}$,\quad\ \ 
$\displaystyle\operatorname{deg}_{\bar\beta}\{\bar\omega_{(\alpha_1\dots\alpha_j;\bar\beta_0\dots\bar\beta_k)}\}
=\sum_{\nu=0}^k\delta_{\bar\beta\bar\beta_\nu}$,
\smallbreak\quad$\displaystyle\operatorname{deg}_\alpha\{\Xi_{(\alpha_1\dots\alpha_e;\bar\beta_1\dots\bar\beta_f)}\}
=\sum_{\nu=1}^e\delta_{\alpha\alpha_\nu}$,\quad\ \
$\displaystyle\operatorname{deg}_{\bar\beta}\{\Xi_{(\alpha_1\dots\alpha_e;\bar\beta_1\dots\bar\beta_f)}\}
=\sum_{\nu=1}^f\delta_{\beta\beta_\nu}$.
\medbreak\noindent As with the weight, we extend the degree by summing over the variables
comprising $A$. If $A$ has the form given in Equation~(\ref{E4.a}), then
\begin{eqnarray*}
\operatorname{deg}_\star(A)&=&\sum_{i=1}^a\operatorname{deg}_\star\{g_{(U_i;\bar V_i})\}
+\sum_{i=a+1}^{a+b}\operatorname{deg}_\star\{h_{(p_i\bar q_i;U_{a+b};\bar V_{a+b})}\}\\
&+&\sum_{i=a+b+1}^{a+b+c}\operatorname{deg}_\star\{\omega_{(U_i;\bar V_i)}\}
+\sum_{i=a+b+c+1}^{a+b+c+d}\operatorname{deg}_\star\{\bar\omega_{(U_i;\bar V_i)}\}\\
&+&\operatorname{deg}_\star\{\Xi_{(U_{a+b+c+d+1};\bar V_{a+b+c+d+1})}\}.
\end{eqnarray*}
If an index does not appear in a monomial, 
we set the degree to zero.

\subsection{Product with a flat torus} Let $\mathbb{T}$ be the flat 2-dimensional torus with $E$ trivial and 
$\omega=0$. If $\mathcal{N}$ has complex dimension $\mm-1$, we set
$\mathcal{M}=\mathcal{N}\times\mathbb{T}$. 
If $P$ is an invariant local formula in complex dimension $\mm$, then
the natural association $\mathcal{N}\rightarrow\mathcal{M}$ 
defines dually an invariant local formula $r(P)$ in dimension $\mm-1$ so that
$$
r(P)(\mathcal{N})(z_1)=P(\mathcal{N}\times\mathbb{T})(z_1,z_2)\,;
$$
the point $z_2\in\mathbb{T}$ that is chosen is irrelevant since $\mathbb{T}$ is homogeneous.
Restriction defines a linear map
$$
r:\mathcal{P}_{\mm,2n}\rightarrow\mathcal{P}_{\mm-1,2n}\,.
$$

\begin{example}\rm The scalar curvature in dimension $2\mm$ is defined by summing over repeated indices relative to a local
orthonormal frame
$$
\tau=\sum_{i,j=1}^{2\mm}R_{ijji}\,.
$$
The restriction $r(\tau)$ is defined by restricting the range of summation to lie over $1\le i,j\le 2\mm-2$.
\end{example}

\subsection{Algebraic formulation} The restriction map can be defined algebraically. Let $P\in\mathcal{P}_{\mm,2n}$.
It is immediate from the definition that $r(P)=0$ if and only if 
$\operatorname{deg}_{\mm}(A)\ge1$ for every monomial $A$ of $P$. 
Since we can permute the indices, if $P$
is invariant, we have that:

\begin{lemma}\label{L5.2} If $P\in\mathcal{P}_{\mm,2n}$, then 
$r(P)=0$ if and only if  $\operatorname{deg}_\alpha(A)>0$ for every monomial $A$ of $P$
and every index $1\le\alpha\le\mm$.
\end{lemma}

It is immediate that the heat trace invariants of the Dolbeault
complex vanish on $\mathbb{T}$. Consequently, Equation~(\ref{E2.c}) yields
\begin{lemma}\label{L5.3}
$r(a_{m,2n}(\mathcal{E}_{\operatorname{Dol}}))=0$.
\end{lemma}

\section{Invariance Theory}\label{S6}
The coordinate and frame normalizations 
of Lemma~\ref{L3.1} are invariant under 
the action
of the unitary groups $U(\mm)$ and $U(\dim(E))$; although invariance under $U(\dim(E))$ will play no direct
role in our analysis, it was central to the proof of Theorem~\ref{T1.3}.
We exploit unitary invariance to show

\begin{lemma}\label{L6.1} Let $0\ne P\in\mathcal{P}_{\mm,2n}$. If $A$ is a monomial of $P$, then
\smallbreak\centerline{$\operatorname{deg}_\alpha(A)=\operatorname{deg}_{\bar\alpha}(A)$ for any $\alpha$.}
\end{lemma}

\begin{proof} Since $P$ is invariant, we can permute the indices. Thus we may suppose that $\alpha=1$.
Make a unitary change of coordinates to define a new holomorphic coordinate
system $\vec w$ so that
$$
\frac{\partial}{\partial w^\alpha}=\left\{\begin{array}{cl}
e^{\sqrt{-1}\theta}\frac{\partial}{\partial z^1}&\text{if }\alpha=1\\[0.05in]
\frac{\partial}{\partial z^\alpha}&\text{if }\alpha>1\end{array}\right\}\text{ and }
\frac{\partial}{\bar\partial w^{\bar\beta}}=\left\{\begin{array}{cl}
e^{-\sqrt{-1}\theta}\frac{\partial}{\partial\bar z^{\bar 1}}&\text{if }\beta=1\\[0.05in]
\frac{\partial}{\partial\bar z^{\bar\beta}}&\text{if }\beta>1\end{array}\right\}\,.
$$
To compute $P^w$, we formally replace the index
$1$ by $e^{\sqrt{-1}\theta}\cdot 1$ and $\bar 1$ by $e^{-\sqrt{-1}\theta}\cdot\bar 1$, and we leave the remaining
indices unchanged to expand each monomial of $P$ multi-linearly.
Thus $A^w=e^{\sqrt{-1}(\operatorname{deg}_1A-\operatorname{deg}_{\bar 1}A)\theta}A$ so
$$
P=\sum_Ac(A,P)A\text{ and }
P^w=\sum_Ae^{\sqrt{-1}(\operatorname{deg}_1(A)-\operatorname{deg}_{\bar 1}(A))\theta}\ 
c(A,P)A\,.
$$
Since $P$ is invariant, $P^w=P$ and thus $\operatorname{deg}_1(A)=\operatorname{deg}_{\bar 1}(A)$
if $c(A,P)\ne0$. 
\end{proof}

\begin{definition}\label{D6.2}\rm
Let $|\xi|^2+|\eta|^2=1$. Make a unitary change of coordinates to define a new holomorphic coordinate
system $\vec w$ so that
$$
\frac{\partial}{\partial w^\alpha}=\left\{\begin{array}{cl}
\xi\frac{\partial}{\partial z^1}+\eta\frac{\partial}{\partial z^2}&\text{if }\alpha=1\\[0.05in]
-\bar\eta\frac{\partial}{\partial z^1}+\bar\xi\frac{\partial}{\partial z^2}&\text{if }\alpha=2\\[0.05in]
\frac{\partial}{\partial z^\alpha}&\text{if }\alpha>2\end{array}\right\}\,,
$$
$$
\frac{\partial}{\bar\partial w^{\bar\beta}}=\left\{\begin{array}{cl}
\bar\xi\frac{\partial}{\partial\bar z^{\bar 1}}
+\bar\eta\frac{\partial}{\partial\bar z^{\bar2}}&\text{if }\beta=1\\[0.05in]
-\eta\frac{\partial}{\partial\bar z^{\bar 1}}
+\xi\frac{\partial}{\partial\bar z^{\bar2}}&\text{if }\beta=2\\[0.05in]
\frac{\partial}{\partial\bar z^{\bar\beta}}&\text{if }\beta>2\end{array}\right\}\,.
$$
If $P$ is a polynomial, let $P^w$ be the expression of $P$ in this new coordinate system.
We formally replace each index
$$
1\rightarrow\xi\cdot 1+\eta\cdot 2,\quad
2\rightarrow-\bar\eta\cdot1+\bar\xi\cdot2,\quad
\bar1\rightarrow\bar\xi\cdot\bar 1+\bar\eta\cdot\bar 2,\quad
\bar 2\rightarrow-\eta\bar 1+\xi\cdot\bar 2
$$
and leave the remaining indices unchanged. We then expand multilinearly to compute $P^w$. 
Of course, the use of the indices `1' and `2' is intended to be illustrative
only, any pair of distinct indices would suffice.
\end{definition}

\begin{definition}\label{D6.3}\rm
If $B$ is a monomial,  
let $\mathcal{B}(B)$ be the set of all monomials $A$ so that changing
a single index $1\rightarrow2$ or $\bar2\rightarrow\bar1$ in $A$ yields $B$; alternatively, so that
$A$ arises by changing a single index $2\rightarrow1$ or $\bar1\rightarrow\bar2$ in $B$.
Let $P^w$ be the expression of a polynomial $P$ in the new coordinate system
given  in Definition~\ref{D6.2} by taking $\xi=\cos(\phi)$ and  $\eta=\sin(\phi)e^{\sqrt{-1}\theta}$.
\begin{equation}\begin{array}{l}\label{E6.a}
P^w=c_0(B,P)\cos(\phi)^{u-1}\sin(\phi)e^{\sqrt{-1}\theta}B+\text{other terms where}\\[0.05in]
u:=\operatorname{deg}_1(B)+\operatorname{deg}_2(B)+
\operatorname{deg}_{\bar 1}(B)+\operatorname{deg}_{\bar 2}(B)\text{ and}\\[0.05in]
c_0(B,P)=\sum_{A\in\mathcal{B}(B)}\nu(A)c(A,P)\text{ where }\nu(A)\ne0\,.
\end{array}\end{equation}\end{definition}

\begin{example}\rm  If $B=g_{(12;\bar 1\bar 2)}g_{(12;\bar 1\bar 1)}$, then $\mathcal{B}(B)=\{A_1,A_2,A_3,A_4\}$
where we have marked with $\star$ the index $1\rightarrow 2$ or $\bar 2\rightarrow\bar1$
that was changed in $A$ to create $B$; we apply the symmetries of Equation~(\ref{E3.b}) to obtain:
\medbreak\quad$A_1=g_{(1^\star1;\bar 1\bar 2)}g_{(12;\bar 1\bar 1)}\ \Rightarrow \ 
B=g_{(2^\star1;\bar 1\bar 2)}g_{(12;\bar 1\bar 1)}\text{ by }1^\star\rightarrow2^\star$,
\smallbreak\quad$A_1=g_{(11^\star;\bar 1\bar 2)}g_{(12;\bar 1\bar 1)}\ \Rightarrow \ 
B=g_{(12^\star;\bar 1\bar 2)}g_{(12;\bar 1\bar 1)}\text{ by }1^\star\rightarrow2^\star,
\ \nu(A_1)=2$,
\smallbreak\quad$A_2=g_{(12;\bar 1\bar 2)}g_{(1^\star1;\bar 1\bar 1)}\ \Rightarrow\ 
B=g_{(12;\bar 1\bar 2)}g_{(2^\star1;\bar 1\bar 1)}\text{ by }1^\star\rightarrow2^\star$,
\smallbreak\quad$A_2=g_{(12;\bar 1\bar 2)}g_{(11^\star;\bar 1\bar 1)}\ \Rightarrow\ 
B=g_{(12;\bar 1\bar 2)}g_{(12^\star;\bar 1\bar 1)}\text{ by }1^\star\rightarrow2^\star,\ \nu(A_2)=2$,
\smallbreak\quad$A_3=g_{(12;{\bar 2^\star}\bar 2)}g_{(12;\bar1\bar1)}\ \Rightarrow\  
B=g_{(12;\bar 1^\star\bar 2)}g_{(12;\bar 1\bar 1)}\text{ by }\bar2^\star\rightarrow\bar1^\star$,
\smallbreak\quad$A_3=g_{(12;\bar 2\bar 2^\star)}g_{(12;\bar1\bar1)}\ \Rightarrow\  
B=g_{(12;\bar 2\bar 1^\star)}g_{(12;\bar 1\bar 1)}\text{ by }\bar2^\star\rightarrow\bar1^\star,
\ \nu(A_3)=-2$,
\smallbreak\quad$A_4=g_{(12;\bar 1\bar2)}g_{(12;\bar 1\bar 2^\star)}\ \Rightarrow\ 
B=g_{(12;\bar 1\bar 2)}g_{(12;\bar 1\bar 1^\star)}\text{ by }\bar2^\star\rightarrow\bar1^\star,
\ \nu(A_4)=-1$,
\smallbreak\quad$u=8\text{ and }c_0(B,P)=2c(A_1,P)+2c(A_2,P)-2c(A_3,P)-c(A_4,P)$.\end{example}

\begin{lemma}\label{L6.5} Let $0\ne P\in\mathcal{P}_{\mm,2n}$.
If $B$ is any monomial, then either no monomial of $\mathcal{B}(B)$ appears in $P$
or at least two monomials of $\mathcal{B}(B)$ appear in $P$.\end{lemma}

\begin{proof} If $A\in\operatorname{\mathcal{B}}(B)$, then 
$\operatorname{deg}_1(B)-\operatorname{deg}_{\bar 1}(B)
=\operatorname{deg}_1(A)-\operatorname{deg}_{\bar 1}(A)-1$.
Thus if $B$ is a monomial of $P$, no monomial of $\mathcal{B}(B)$ is a monomial of $P$ by
Lemma~\ref{L6.1}
and Lemma~\ref{L6.5} follows.
We therefore assume $B$ is not a monomial of $P$ and thus
$c_0(B,P)=0$. We use Equation~(\ref{E6.a}). Since the multiplicities $\nu(A)$ are non-zero integers, if 
$c(A_1,P)\ne0$ for some $A_1$, there must exist at least another monomial $A_2\in\mathcal{B}(B)$ to cancel
off $\nu(A_1)c(A_1,P)$ in Equation~(\ref{E6.a}).
\end{proof}

We use Lemma~\ref{L6.5} as to prove the following result.

\begin{lemma}\label{L6.6}
Let $0\ne P\in\ker(r)\cap\mathcal{P}_{\mm,2n}$ where $n\le\mm$.
There exists a monomial $A$ of $P$ of the form
given in Equation~(\ref{E4.a}) so that $U_\nu=(\nu,\dots,\nu)$ for $\nu\le n$.
Furthermore, $\ell(A)\ge\mm$, and if $\ell(A)=\mm$, then none of the $U_\nu$ is empty.
\end{lemma}

\begin{proof} We shall apply Lemma~\ref{L6.1} and Lemma~\ref{L6.5}.
\smallbreak\noindent{\bf Step 1:} We concentrate on the collections $(U_\nu;\bar V_\nu)$ and
suppress the particular variables $g$, $h$, $\omega$, $\bar\omega$ or $\Xi$ in which they
appear. Choose a monomial $A=(U_1;\bar V_1)A_0$ of $P$ 
such that $\operatorname{deg}_1(U_1)$ is maximal. If $U_1=(1\dots1)$, we proceed to Step 2.  Let
$U_1=(1\dots 1\alpha^\star\dots)$ for $\alpha\ne1$. By permuting the indices we may assume $\alpha=2$. Set 
$B=(1\dots 11^\star\dots;\bar V_1)A_0$.
Use Lemma~\ref{L6.5} to choose a monomial $A_1\in\mathcal{B}(B)$ of $P$ different from $A$. Since $A_1\ne A$,
$A_1$ does not transform to $B$ by changing an index of $U_1$ and thus 
$A_1$ has an index collection $\tilde U_1$ with one more occurrence of the index `1' which contradicts the 
maximality of $A$. This contradiction shows $U_1=(1\dots1)$.

\smallbreak\noindent{\bf Step 2:} Choose $A=(1\dots 1;\bar V_1)(U_2;\bar V_2)A_0$ so the number
of occurrences of the index $2$ in $U_2$ is maximal. If $U_2=(2\dots2)$ proceed to Step 3.
Otherwise assume $U_2=(2\dots2\alpha^\star\dots)$ for $\alpha\ne2$. 
Let $B=(1\dots 1;\bar V_1)(2\dots 22^\star\dots;\bar V_2)A_0$
be obtained by changing the index $\alpha$
 to the index $2$.  By Lemma~\ref{L6.5} (where we replace the indices
$(1,2)$ by $(2,\alpha)$),
we can choose $A_1\in\mathcal{B}(B)$ to be a monomial of $P$ different from $A$. 
Changing the index $\alpha\rightarrow 2$ does not
affect $U_1=(1\dots1)$. Since $A_1\ne A$, it has an index collection $\tilde U_2=(2\dots 22\dots)$ which contradicts the maximality of $A$.
Thus we can choose $A=(1\dots 1;\bar V_1)(2\dots 2;\bar V_2)A_0$.

\smallbreak\noindent{\bf Step 3:} We continue in this fashion to construct $A$ with the desired form.
The process stops when $\nu=\mm$ or when $\nu=\ell(A)$. Since $\operatorname{deg}_\nu(A)\ne0$ for $1\le\nu\le\mm$
by Lemma~\ref{L5.2}, we have $\ell(A)\ge\mm$. 
If $\ell(A)=\mm$, none of the $U_\nu$ could be empty or the index $\nu$ would not appear in $A$.\end{proof}

\section{The proof of Theorem~\ref{T1.4}}\label{S7}
The subalgebra
of variables of weight 2 will play a distinguished role and we set
$$
\mathfrak{B}_\mm=\mathbb{C}[g_{(\alpha_0\alpha_1;\bar\beta_0\bar\beta_1)},h_{(p\bar q;\alpha_1;\bar\beta_1)},
\omega_{(\alpha_0;\bar\beta_1)},\bar\omega_{(\alpha_1;\bar\beta_0)}]\,.$$ By Lemma~\ref{L5.3},
$a_{m,2n}(\mathcal{E}_{\operatorname{Dol}})\in\ker(r)$.
Consequently, the fact that $a_{m,2n}(\mathcal{E}_{\operatorname{Dol}})=0$ for $2n<m$
will follow from the following result.
\begin{lemma}\label{L7.1} Let $2n\le m$.
\begin{enumerate}
\item $\ker(r)\cap\mathcal{P}_{\mm,2n}=\{0\}$ if $2n<m$.
\item $\ker(r)\cap\mathcal{P}_{\mm,m}\subset\mathfrak{B}_m\oplus\bigoplus_{\alpha_1\bar\beta_1}\omega_{\alpha_1}\bar\omega_{\bar\beta_1}\mathfrak{B}_m$.
\end{enumerate}\end{lemma}
\begin{proof} Let $P\in\ker(r)\cap\mathcal{P}_{\mm,2n}$ where $2n\le m$. 
Choose a monomial $A$ of $P$ satisfying the conclusions of Lemma~\ref{L6.6}. We then have $\ell(A)\ge\mm$ 
and, by Lemma~\ref{L5.2}, 
$\operatorname{deg}_\alpha(A)\ne0$ for all $\alpha$. We examine $\Xi=\Xi_{(U;\bar V)}$
where $|U|=e$ and $|V|=f$; the role of the variables of weight 1 is crucial. 
\smallbreak\noindent{\bf Case 1. Suppose $e=0$.} Then $\deg_\alpha(A)=0$ for $\alpha>a+b+c+d$. Since $r(P)=0$,
$\operatorname{deg}_{\mm}(A)\ne0$. Thus $a+b+c+d\ge\mm$.  The normalizations of Lemma~\ref{L3.1} show
$$
m\ge2n=\operatorname{weight}(A)\ge 2a+2b+2e+2d+f\ge 2a+2b+2c+2d\ge2\mm=m\,.
$$
Thus equality holds. This implies $2n=m$, $f=0$, and $P$ is a polynomial in the variables of weight 2.
\smallbreak\noindent{\bf Case 2. Suppose $f=0$.} 
A similar argument
using the anti-holomorphic indices shows $2n=m$,  $e=0$, and $P$ is a polynomial in the variables of weight 2.
\smallbreak\noindent{\bf Case 3: Suppose $e>0$ and $f>0$.} If $a+b+c+d+1<\mm$, we could choose $A$
so that $\deg_\mm A=0$ which is false. Consequently, $a+b+c+d+1\ge\mm$. We estimate
\begin{eqnarray*}
m&\ge&2n=\operatorname{weight}(A)\ge 2a+2b+2c+2d+e+f\\
&\ge&2a+2b+2e+2d+2\ge 2\mm=m\,.
\end{eqnarray*}
Thus all the inequalities are in fact equalities. This implies $2n=m$, $e=f=1$, and the remaining variables comprising $A$ all have weight 2.
\medbreak If a $\bar\omega_*$ variable of weight 2 does not contain a holomorphic index, 
Lemma~\ref{L6.6} shows $\operatorname{deg}_\nu(A)=0$ for some holomorphic index which is
false since $r(A)=0$. Since Lemma~\ref{L6.6} also holds for the anti-holomorphic indices, 
if any of the $\omega_*$ variables of weight 2 does not contain an anti-holomorphic index,
then $\operatorname{deg}_{\bar\nu}(A)=0$ for some anti-holomorphic index which is false. Thus all the variables of weight 2 which divide $A$ belong to $\mathfrak{B}_m$ and Assertion~(2) holds.
\end{proof}

Let $\Re(\cdot)$ and $\Im(\cdot)$ be the real and imaginary parts of a $1$-form.

\begin{lemma}\label{L7.2}
If $\mathcal{M}$ is a Riemann surface, then
$a_{2,2}(\mathcal{M})
=\star\operatorname{Todd}_2(\mathcal{M})+\displaystyle\frac{d\Im\omega}\pi$.
\end{lemma}

\begin{proof}
\'Alvarez L\'opez and Gilkey \cite{AG20} showed that 
$a_{2,2}(\mathcal{M})=\frac\tau{8\pi}-\frac{1}{\pi}\delta(\Re(\omega))$. 
We have $\frac\tau{8\pi}=\star\operatorname{Todd}_2(\mathcal{M})$ by Hirzebruch~\cite{H66}.
Let 
\smallbreak\qquad\quad
$\omega=(u+\sqrt{-1}v)(dx+\sqrt{-1}dy)=(udx-vdy)+\sqrt{-1}(udy+vdx)$;
\smallbreak\qquad\quad
$-\delta(\Re(\omega))=-\delta(udx-vdy)=u_x-v_y$, and
\smallbreak\qquad\quad $d(\Im(\omega))=d(udy+vdx)=(u_x-v_y)dx\wedge dy$.
\end{proof}

We must improve Lemma~\ref{L6.6}.

\begin{lemma}\label{L7.3} Let $0\ne P\in\ker(r)\cap\mathcal{P}_{\mm,m}$. 
There exists a monomial $A$ of $P$ so
$$\begin{array}{l}
A=g_{(11;\bar\beta_1\bar\beta_2)}\dots g_{(aa;\bar\beta_{2a-1}\bar\beta_{2a})}
h_{(p_{a+1}\bar q_{a+1};a+1;\overline{a+1})}\dots h_{(p_b\bar q_b;a+b;\overline{a+b})}\\[0.08in]
\qquad\omega_{(a+b+1;\overline{a+b+1})}\dots\omega_{(a+b+c;\overline{a+b+c})}
\bar\omega_{(a+b+c+1;\overline{a+b+c+1})}\\[0.08in]
\qquad\dots\bar\omega_{(a+b+c+d;\overline{a+b+c+d})}\Xi\\[0.08in]
\text{ where }\Xi=1\text{ or }\Xi=\Xi_{(\mm;\bar\mm)}\text{ and }\beta_\nu\le a\text{ for }1\le\nu\le 2a\,.
\end{array}$$
\end{lemma}

\begin{proof} Apply Lemma~\ref{L6.6} and Lemma~\ref{L7.1}
to choose a monomial $A$ of $P$ of the form
\begin{equation}\label{E7.a}\begin{array}{l}
A=g_{(11;\bar\beta_1\bar\beta_2)}\dots g_{(aa;\bar\beta_{2a-1}\bar\beta_{2a})}
h_{(p_{a+1}\bar q_{a+1};a+1;\bar\beta_{a+1})}\dots h_{(p_b\bar q_b;a+b;\bar\beta_{a+b})}\\[0.05in]
\qquad\omega_{(a+b+1;\bar\beta_{a+b+1})}\dots\omega_{(a+b+c;\bar\beta_{a+b+c})}
\bar\omega_{(a+b+c+1;\bar\beta_{a+b+c+1})}\\[0.08in]
\qquad\dots\bar\omega_{(a+b+c+d;\bar\beta_{a+b+c+d})}\Xi\text{ where }\Xi=1\text{ or }\Xi_{(\mm;\bar\beta_m)}\,.
\end{array}\end{equation}
We use Lemma~\ref{L6.1} and Equation~(\ref{E7.a}) to see
$$
a<\beta\quad\Leftrightarrow\quad\operatorname{deg}_\beta(A)=1\quad\Leftrightarrow
\quad\operatorname{deg}_{\bar\beta}(A)=1\,.
$$
We say that an anti-holomorphic index $\bar\beta$ touches an anti-holomorphic index $\bar\gamma$ in $A$ 
if $A$ is divisible by $g_{(\alpha\alpha;\bar\beta\bar\gamma)}$ for some $\alpha$; we say that $\bar\beta$
touches itself in $A$ if we can take $\bar\beta=\bar\gamma$. 

Let $a<\beta$ so $\operatorname{deg}_\beta(A)=\operatorname{deg}_{\bar\beta}(A)=1$.
Let $\gamma\ne\beta$. We construct a monomial $B$ by replacing $\bar\beta$ by $\bar\gamma$;
$\operatorname{deg}_{\bar\beta}B=0$ and
$A$ is obtained from $B$ by changing $\bar\gamma\rightarrow\bar\beta$.
We apply Lemma~\ref{L6.5} to find
$A_1\in\mathcal{B}(B)$ which appears in $P$ with $A\ne A_1$. Since $\operatorname{deg}_{\bar\beta}B=0$
and $\operatorname{deg}_{\bar\beta}(A_1)\ne0$, $A_1$ is obtained from $B$ by changing
$\bar\gamma\rightarrow\bar\beta$ or, equivalently, $A_1$ arises from $A$ by
interchanging a $\bar\beta$ with a $\bar\gamma$ index. Thus, in particular, since $A_1\ne A$,
two anti-holomorphic indices of degree 1 in $A$ can not touch in $A$.

Choose $A$ of the form given in Equation~(\ref{E7.a}) so the number of anti-holomorphic
indices which touch themselves in $A$ is maximal.
Suppose $\operatorname{deg}_{\bar\beta}(A)=1$ and
$\bar\beta$ touches another anti-holomorphic index $\bar\gamma$ in $A$. 
Then there is a monomial $A_1$ of $P$ different
from $A$ defined by interchanging $\bar\beta$ and $\bar\gamma$. This is not possible since $\bar\gamma$
would touch itself in $A_1$ which contradicts the maximality of $A$. Thus $\{\bar V_1,\dots,\bar V_a\}$
consists solely of anti-holomorphic indices of degree $2$ in $A$ and hence must comprise all the anti-holomorphic
indices of degree 2 in $A$. 

We say that a holomorphic index $\alpha$ touches an anti-holomorphic index $\bar\beta$ in $A$ if $A$
is divisible by $\omega_{(\alpha;\bar\beta)}$, by $\bar\omega_{(\alpha;\bar\beta)}$ or by $\Xi(\alpha;\bar\beta)$.
If $A$ is as constructed above, then every holomorphic index $\alpha$ of degree $1$ 
touches an anti-holomorphic index $\bar\beta$ of degree $1$ in $A$. 
Among all the monomials $A$ constructed above, choose $A$ so the number of holomorphic indices $\alpha$
which touch $\bar\alpha$ in $A$ is maximal. Let $a+b<\alpha$. If $\alpha$ touches $\bar\beta$ in $A$ with
$\alpha\ne\beta$, we can interchange $\bar\beta$ and $\bar\alpha$ to construct a monomial $A_1$ of $P$
where there is one more holomorphic - anti-holomorphic touching which is impossible. 
Therefore $A$ has the form given in the Lemma.
\end{proof}

We use Lemma~\ref{L7.1} to improve Equation~(\ref{E2.c}). Let $\mathcal{M}=\mathcal{M}_1\times\mathcal{M}_2$.
Since $a_{m_i,2n_i}=0$ for $2n_i<m_i$, taking $2n=m$ in Equation~(\ref{E2.c}) yields
\begin{equation}\label{E7.b}
a_{m,m}(\mathcal{E}_{\operatorname{Dol}}(\mathcal{M}))(z_1,z_2)=
a_{m_1,m_1}(\mathcal{E}_{\operatorname{Dol}}(\mathcal{M}_1))(z_1)
\cdot a_{m_2,m_2}(\mathcal{E}_{\operatorname{Dol}}(\mathcal{M}_2))(z_2)\,.
\end{equation}

Let $\mathcal{N}_\kk(0)=(N^{2\kk},J_N,g_N,E_N,h_N,0)$ 
be a structure of complex dimension $\kk$ with trivial twisting (1,0)--form $\omega_0=0$. 
Let
$$\mathbb{T}^2(\omega_i)=(S^1\times S^1,dx^2+dy^2,J_2,\text{\bf1},h_0,\omega_i)
$$
 be the torus $S^1\times S^1$
with the flat metric $dx^2+dy^2$, 
usual complex structure $J_2:\frac{\partial}{\partial x}\rightarrow\frac{\partial}{\partial y}$ and 
$J_2:\frac{\partial}{\partial y}\rightarrow-\frac{\partial}{\partial x}$,
flat line bundle {\bf1}$=(S^1\times S^1)\times\mathbb{C}$, trivial
Hermitian metric $h_0$, and (possibly) non-trivial
twisting (1,0)-form $\omega_i$ with $\partial\omega_i=0$. 
Let $\vec\omega=(\omega_1,\dots,\omega_{\mm-\kk})$. Set
$$
\mathcal{M}(\kk;\vec\omega)
:=\mathcal{N}_\kk(0)\times\mathbb{T}(\omega_1)\times\dots\times\mathbb{T}(\omega_{\mm-\kk})\,.
$$
By Equation~(\ref{E7.b}),
\begin{eqnarray*}
&&a_{m,m}(\mathcal{E}_{\operatorname{Dol}}(\mathcal{M}(\kk;\vec\omega)))\\
&=&a_{k,2\kk}(\mathcal{E}_{\operatorname{Dol}}(\mathcal{N}_\kk))a_{2,2}(\mathcal{E}_{\operatorname{Dol}}(\mathbb{T}^2(\omega_1)))\dots 
a_{2,2}(\mathcal{E}_{\operatorname{Dol}}(\mathbb{T}^2(\omega_{\mm-\kk})))\,.
\end{eqnarray*}
By Lemma~\ref{L7.3}, if $0\ne P\in\ker(r)\cap\mathcal{P}_{\mm,m}$, then $P(\mathcal{M}(\kk;\vec\omega))\ne0$
for some $\kk$ and some $\vec\omega$. On the other hand, we may use Equation~(\ref{E7.b}),
Theorem~\ref{T1.3},
and Lemma~\ref{L7.2}
to see that
$$\left\{a_{m,m}-\left\{\operatorname{Td}\wedge\operatorname{ch}\wedge\Theta
\right\}_{m}\right\}(\mathcal{M}(\kk;\vec\omega))=0\text{ for all }\kk\text{ and }\vec\omega\,.
$$
Theorem~\ref{T1.4} now follows.\hfill\qed

\subsection{The kernel of $r$} It seems useful to identify $\ker(r:\mathcal{P}_{\mm,2\mm}\rightarrow\mathcal{P}_{\mm-1,2\mm})$
in a bit more detail.
Let $\operatorname{ch}_k$ be the $k^{\operatorname{th}}$ component of
the Chern character (see~\cite{H66}). We decompose the graded ring of characteristic forms into homogeneous components
$$
\mathcal{C}_m:=\mathbb{C}[\operatorname{ch}_k(TM,J,g),\operatorname{ch}_k(E,h)]=\oplus_k\mathcal{C}_m^{2k}\,.
$$
We also consider the graded ring
$$
\mathcal{D}_m:=\mathbb{C}[\operatorname{ch}_k(TM,J,g),\operatorname{ch}_k(E,h),d\omega,d\bar\omega,\omega,\bar\omega]
=\oplus_k\mathcal{D}^{2k}_m\,.
$$
Let $\mathcal{P}_{\mm,2\mm}^{g,E}$ be the
subspace of invariants which are independent of $\omega$.

\begin{lemma} 
\ \begin{enumerate}
\item $\ker(r:\mathcal{P}_{\mm,m}^{g,E}\rightarrow\mathcal{P}_{\mm-1,m}^{g,E})=\star\mathcal{C}_m^{m}$.
\item $\ker(r:\mathcal{P}_{\mm,m}\rightarrow\mathcal{P}_{\mm-1,m})=\star\mathcal{D}_m^m$.
\end{enumerate}\end{lemma}

\begin{proof} Assertion~(1) follows from Theorem~1 of Gilkey~\cite{G73a}. Assertion~(2) is a scholium to the arguments we
have given above. We use Assertion~(1) to control the metric terms. We express $P\in\ker(r)$ on $\mathcal{M}(\kk;\vec\omega)$
as a metric invariant on $\mathcal{M}_{\kk}$ times invariants on the $\mathbb{T}^2(\omega_i)$. The metric invariant is itself
in the kernel of $r$ and hence is a characteristic class; the remaining invariants only involve $\Theta$. Thus there is an
element of $\mathcal{D}$ which hits this invariant and hence by Lemma~\ref{L7.3}, $P$ can be decomposed appropriately.
\end{proof}


\begin{thebibliography}{99}
\bibitem{AG20} {\it J. \'Alvarez L\'opez and P. Gilkey}: The local index density of the perturbed de~Rham complex. 
arXiv 2004.02243.
\bibitem{AKL} {\it J. A. \'Alvarez L\'opez, Y. A. Kordyukov, and E. Leichtnam}: A trace formula for foliated flow. In preparation.
\bibitem{ABP73} {\it M. F. Atiyah, R. Bott, and V. K. Patodi}: On the heat equation and the index theorem.
Invent. Math. {\bf13} (1973), 279--330; errata {\bf 28} (1975), 277--280. Zbl~0257.58008, 
MR650828, doi:10.1007/BF01425417; Zbl~0301.58018, MR650829, doi:10.1007/BF01425562
\bibitem{BZ92} {\it J.-M. Bismut, and W. Zhang}: An extension of a theorem by Cheeger and Muller. Ast\'erisque, Vol. 205, Soc. Math. France (1992). Zbl~0781.58039, MR1185803
\bibitem{BF97} {\it M. Braverman, and M. Farber}: Novikov type inequalities for differential forms with non-isolated zeros. 
Math. Proc. Camb. Phil. Soc. {\bf 122} (1997), 357--375. Zbl~0894.58012, MR1458239, doi:10.1017/S0305004197001734
\bibitem{BH01} {\it D.~Burghelea and S.~Haller}: On the topology and analysis of a closed one form. I. (Novikov's theory revisited). 
Essays on geometry and related topics, Vol. 1, 2, Monogr. Enseign. Math., Vol.~38, Enseignement Math., Geneva, 2001, pp.~133--175. Zbl~1017.57013, MR1929325
\bibitem{BH08} {\it D.~Burghelea, and S.~Haller}: Dynamics, Laplace transform and spectral geometry. J. Topol. {\bf 1} (2008), 115--151. Zbl~1156.57022, MR2365654, doi:10.1112/jtopol/jtm005
\bibitem{C44} {\it S. Chern}: A simple intrinsic proof of the Gauss-Bonnet formula for closed Riemannian manifolds. 
Ann. Math. (2) {\bf 45} (1944), 741--752. Zbl~0060.38103, MR11027, doi:10.2307/1969302
\bibitem{G73} {\it P. Gilkey}: Curvature and the eigenvalues of the Laplacian for elliptic complexes. Adv. Math. {\bf 10} (1973), 344--382. Zbl~0259.58010,MR324731, doi:10.1016/0001-8708(73)90119-9
\bibitem{G73a} {\it P. Gilkey}: Curvature and the eigenvalues of the Laplacian for K\"ahler manifolds. Adv. Math. {\bf 11} (1973), 311--325. Zbl~0285.53044, MR334290, doi:10.1016/0001-8708(73)90014-5
\bibitem{G95} {\it P. Gilkey}: Invariance Theory, the Heat Equation, 
and the Atiyah-Singer Index Theorem $2^{\operatorname{nd}}$ ed., 
Studies in Advanced Mathematics, CRC Press (1995). Zbl~0856.58001, MR1396308
\bibitem{GNP97} {\it P. Gilkey, S. Nik\v cevi\'c, and J. Pohjanpelto}: The local index formula for a Hermitian manifold.
Pac. J. Math. {\bf180} (1997), 51--56. Zbl~0885.58091, MR1474893, doi:10.2140/pjm.1997.180.51
\bibitem{HM06} {\it F.~R.~Harvey and G.~Minervini}: Morse Novikov theory and cohomology with forward supports. Math. Ann. {\bf 335} (2006), 787--818. Zbl~1109.57019, MR2232017, doi:10.1007/s00208-006-0765-4
\bibitem{HS85} {\it B. Helffer, and J. Sj\"ostrand}: Puits multiples en mecanique semi-classique~IV: Edude du complexe de Witten. Commun. Partial Differ. Equations 
{\bf 10} (1985), 245--340. Zbl~0597.35024, MR780068, doi:10.1080/03605308508820379
\bibitem{H66} {\it F. Hirzebruch}: Topological methods in algebraic geometry. Springer-Verlag (Berlin) 1966. 
Zbl~0138.42001, MR0202713
\bibitem{Mi15} {\it G.~Minervini}: A current approach to Morse and Novikov theories. Rend. Mat. Appl., VII. Ser. {\bf 36} (2015), 95--195. MR3533253, Zbl~1361.58007
\bibitem{N81} {\it S. P. Novikov}: Multivalued functions and functionals. An analogue of the Morse theory. Soviet Math., Dokl. {\bf 24} (1981), 222--226. Zbl~0505.58011, MR630459
\bibitem{N82} {\it S. P. Novikov}: The Hamiltonian formalism and a multivalued analogue of Morse theory. Russian Math. Surveys {\bf 37} (1982), 1--56. MR676612
\bibitem{P71} {\it V. K. Patodi}: Curvature and the eigenforms of the Laplace Operator. J. Differ. Geom. {\bf 5} (1971), 233--249. Zbl~0211.53901, MR292114, doi:10.4310/jdg/1214429791
\bibitem{P71a} {\it V. K. Patodi}: An analytic proof of the Riemann-Roch-Hirzebruch theorem for K\"ahler manifolds. J. Differ. Geom. {\bf 5} (1971), 251--283. Zbl~0219.53054, MR290318, doi:10.4310/jdg/1214429991
\bibitem{Pa87} {\it A. Pazhitnov}: An analytic proof of the real part of Novikov's inequalities. Soviet Math., Dokl. {\bf 35} (1987), 456--457. Zbl~0647.57025, MR891557
\bibitem{S68} {\it R. T. Seeley}: Complex powers of an elliptic operator. Proc. Symp. Pure Math. {\bf10} (1968), Amer. Math. Soc., 288--307. Zbl~0159.15504, MR0237943
\bibitem{W46}  {\it H. Weyl}: The classical groups. Princeton Univ. Press, Princeton (1946) (8${}^{\operatorname{th}}$ printing). Zbl~1024.20502, MR1488158
\bibitem{W82} {\it E. Witten}: Supersymmetry and Morse theory. J. Differ. Geom. {\bf17} (1982), 661-692. Zbl~0499.53056, MR683171, doi:10.4310/jdg/1214437492


\end{thebibliography}
\end{document}